\documentclass[11pt,twoside]{article}

\newif\iffigures
\figurestrue

\usepackage{pdfsync}
\usepackage{multirow} 
\usepackage{enumitem}
\usepackage{amsbsy, amsthm, amsmath, amssymb, amsfonts, mathrsfs, mathtools} 
\usepackage[numbers]{natbib}
\setcitestyle{square,aysep={},yysep={;}}
\usepackage{dsfont,verbatim,fancyhdr}
\usepackage{float} 
\usepackage[hang]{caption} 
\usepackage[usenames,dvipsnames]{xcolor}
\usepackage[bookmarks,bookmarksnumbered,colorlinks=true,pdfstartview=FitV, linkcolor=Red,citecolor=blue!90!black,urlcolor=blue!90!black]{hyperref}
\usepackage{tikz}
\usepackage{tkz-graph}
\usepackage[normalem]{ulem}
\usepgflibrary{arrows}
\usepackage{autonum}

\usepackage{graphicx} 
\usepackage{subcaption}
\usepackage{caption}
\usepackage[ruled,vlined]{algorithm2e}
\usepackage[font=small]{caption}

\usepackage{nicefrac}

\newif\ifpageone

\renewcommand{\baselinestretch}{1.3}
\allowdisplaybreaks

\bibpunct{\textcolor{blue!90!black}{(}}{\textcolor{blue!90!black}{)}}{\textcolor{blue!90!black}{;}}{a}{\textcolor{blue!90!black}{,}}{\textcolor{blue!90!black}{;}}

\newcommand{\red}[1]{\textcolor{red}{#1}}

\newcommand{\bs}[1]{\boldsymbol{#1}}
\def \d {\mathrm{d}}

\def \N {\mathbb{N}}

\def \S {\mathcal{S}}

\def \R {\mathbb{R}}

\def \aa {\bs\alpha}
\def \ss {\bs\sigma}

\def \oo {\mathbf{0}}

\def \mm {\mathbf{m}}
\def \MM {\mathbf{M}}
\def \nn {\mathbf{n}}
\def \NN {\mathbf{N}}

\def \JJ {\mathbf{J}}

\def \xx {\mathbf{x}}

\def \XX {\mathbf{X}}

\def \yy {\mathbf{y}}
\def \YY {\mathbf{Y}}

\newcommand{\set}[1]{\left\{#1\right\}}

\newtheorem{theorem}{Theorem}[section]
\newtheorem{proposition}[theorem]{Proposition}

\newtheorem{definition1}[theorem]{Definition} 
\newtheorem{remark1}[theorem]{Remark} 

\long\def\symbolfootnote[#1]#2{\begingroup\def\thefootnote{\hspace*{-1mm}\fnsymbol{footnote}}\footnote[#1]{#2}\endgroup}

\title{\bf 
Filtering coupled Wright--Fisher diffusions
}
\author{
\normalsize\textsc{Chiara Boetti}\\
\normalsize\emph{University of Bath}\\[3mm]
\normalsize\textsc{Matteo Ruggiero}\\
\normalsize\emph{University of Torino and Collegio Carlo Alberto}}
\date{\today}

\begin{document}
\maketitle
\thispagestyle{empty}

\vspace{-5mm}

\begin{quote}
\small

Coupled Wright--Fisher diffusions have been recently introduced to model the temporal evolution of finitely-many allele frequencies at several loci. These are vectors of multidimensional diffusions whose dynamics are weakly coupled among loci through interaction coefficients, which make the reproductive rates for each allele depend on its frequencies at several loci. Here we consider the problem of filtering a coupled Wright--Fisher diffusion with parent-independent mutation, when this is seen as an unobserved signal in a hidden Markov model. 
We assume individuals are sampled multinomially at discrete times from the underlying population, whose type configuration at the loci is described by the diffusion states, and adapt recently introduced duality methods to derive the filtering and smoothing distributions. These respectively provide the  conditional distribution of the diffusion states given past data, and that conditional on the entire dataset, and are key to be able to perform parameter inference on models of this type. We show that for this model these distributions are countable mixtures of tilted products of Dirichlet kernels, and describe their mixing weights and how these can be updated sequentially. The evaluation of the weights involves the transition probabilities of the dual process, which are not available in closed form. We lay out pseudo codes for the implementation of the algorithms, discuss how to handle the unavailable quantities, and briefly illustrate the procedure with synthetic data. 

\end{quote}

\vfill
\newpage

\section{Introduction}

Wright--Fisher (WF) Markov chains and their diffusion approximations are among the most widely used stochastic models in mathematical biology. See, e.g., \cite{M77,E79,EK86,D93,E09,F10}. Recently they have also been seen as a temporally evolving parameter which becomes the inferential goal in a Bayesian statistical problem; see \cite{CG09,PR14,KKK21,KKK24}. See also \cite{PRS16,ALR21,ALR23} for nonparametric extensions involving Fleming--Viot and Dawson--Watanabe measure-valued diffusions and \cite{ADR24} for related software.

Coupled Wright--Fisher (c-WF) diffusions extend classical WF diffusions to weakly dependent vectors of WF diffusions. 
These vectors model the temporal evolution of allelic frequencies in a multi-locus multi-allelic process when pairwise interactions can occur among different loci. Pairwise interactions are expressed in terms of selection coefficients, leading to weak coupling mechanisms as they affect the reproduction rates of the alleles. This class of models was introduced by \cite{AEK19} as the scaling limit of a Moran-type population of haploid, asexually reproducing individuals with mutation and inter-locus selection, motivated by the evolution of certain populations of bacteria where the dependence between loci can be effectively encoded by a selection mechanism. Setting the inter-locus selection parameters as null yields as a special case a collection of independent WF diffusions with selection as they appear, e.g., in \cite{BEG00}. See also \cite{E04}, Chapters 6-7, and \cite{EN89} for previous contributions related to the multi-locus case.
\cite{FHK21} identified a dual jump process for c-WF diffusions (the notion of duality will be recalled in Section \ref{sec:duality} below), while \cite{GHK21} investigated their exact simulation when mutations are parent-independent and the stationary distribution is available.
Different and stronger coupling mechanisms, which for example produce interacting diffusions via migration of individuals in underlying structured populations, had been previously studied, among others, in \cite{V90,DGV95,DG99, GKW01,GLW05}. See also \cite{PR13} for a Bayesian nonparametric interpretation of a class of interacting Fleming--Viot models through P\'olya urn schemes.

In this paper we study the distributional properties of c-WF diffusions with parent-independent mutations in presence of data collected at discrete times from the underlying population. We consider a hidden Markov model statistical framework \citep{CMR05}, whereby the c-WF diffusion describes a temporally evolving unobserved \emph{signal}, which is the primary object of inference, and we assume incomplete observations are collected at a set of discrete time points from an appropriate distribution which reflects the underlying population structure. In particular, since the data collection reveals partial information on the population composition through a non linear link, this inferential problem does not fall into the realm of Kalman-type filtering \citep{BC09,CMR05}.
Given such data collection framework, we identify the conditional distributions of the diffusion states given past data (\emph{signal prediction}), given past and present data (\emph{filtering}) and given the entire dataset, including data collected at later times (\emph{marginal smoothing}). These distributions constitute the fundamental tools in view of performing parameter inference for a hidden Markov model. The latter, in particular, is often used to improve previous estimates, obtained through filtering, once new data are collected at later times, which typically result in smoother estimates.  We obtain these results by adapting a strategy based on duality proposed by \cite{KKK24}, leveraging on the knowledge of the dual process which was investigated in \cite{FHK21}. The implementation of the above ingredients and recipes is not immediate in that we first have to show that the model at hand satisfies certain requirements, which include showing the model is reversible, then we have to identify the family of likelihoods that allows to perform the envisioned calculations, and finally we have to identify correctly the quantities that appear in the final distributions. Our result provide a description of the recursions that allow to compute the filtering and the smoothing distributions, overcoming the unavailability of the transition function of the c-WF diffusion, which in principle is involved in the evaluation of these distributions. 

The paper is organized as follows. 
In Section \ref{sec: c-WF} we first recall generalities about c-WF diffusions, and then show that, under the assumption of parent-independent mutations, they are reversible with respect to their stationary distribution, which is a product of tilted Dirichlet distributions. We also recall the dual process of these diffusions, which is a multidimensional continuous-time  Markov chain.
In Section \ref{sec: filtering} we formulate the inferential problem of interest as a hidden Markov model driven by a c-WF diffusion and, leveraging on the shown reversibility, we derive the filtering and smoothing distributions when the data are assumed to be collected from products of multinomial densities parameterized by the signal at several discrete time points.
We show that the marginal distributions of c-WF states are mixtures of products of tilted Dirichlet distributions, and we describe how to update the time-dependent mixture weights based on the previous available conditional distribution. This in turn allows to devise algorithms that, starting from an initial distribution arbitrarily chosen within the same class as the stationary distribution, allow to recursively compute the marginal distributions of the process states, conditional on past data or conditionally on the entire dataset. These algorithms are laid out in Section \ref{sec:illustration}, which also presents a brief illustration of the filtering procedure with synthetic data. 
We also we discuss implementation issues, particularly concerning numerical methods for dealing with the dual process transitions, which are not available in closed form, and some other unavailable normalizing constants, through Monte Carlo procedures.



\section{Coupled Wright--Fisher diffusions}\label{sec: c-WF}


\subsection{Generalities}

Define the space
$$\Delta=\times_{l=1}^{L}\Delta_{K_{l}}, \quad 
\Delta_{K_{l}} = \bigg\{\xx\in[0,1]^{K_{l}} : \sum_{i=1}^{K_{l}} x_{i}^{(l)}=1\bigg\},$$
where $L \geq 1$ denotes the number of loci, $K_{1}, \ldots, K_{L}$ are positive integers representing the amount of alleles at each locus and $K = \sum_{l=1}^{L} K_{l}$ is the total number of alleles among all loci. 
Following \cite{FHK21}, a \emph{coupled Wright--Fisher} (c-WF) diffusion is the $\Delta$-valued strong solution of 
\begin{equation}\label{sde}
    \d\XX(t) = \mu(\XX(t)) \d t + D(\XX(t)) \nabla V(\XX(t)) \d t + D^{1/2}(\XX(t))\d\textbf{W}(t),
\end{equation}
where $\mathbf{W} = \left(\mathbf{W}^{(1)}, \ldots, \mathbf{W}^{(L)}\right)^{T}$ is a multi-dimensional Brownian motion, each $\mathbf{W}^{(l)}$ having the same dimension as $\mathbf{X}^{(l)}$. 
Here $\XX(t)$ is a concatenation of the $L$ vectors 
$$	\XX^{(l)}(t)=\Big(X_{1}^{(l)}(t), \ldots, X_{K_{l}}^{(l)}(t)\Big)^T,$$
where $X_{i}^{(l)}(t)$ is the relative frequency of the $i$-th type among those at locus $l$ at time $t$.
In particular, drift and diffusion coefficients are defined as
$$
\mu(\xx)=\left(\begin{array}{c}
	\mu^{(1)}\left(\xx^{(1)}\right) \\
	\vdots \\
	\mu^{(L)}\left(\xx^{(L)}\right)
\end{array}\right), \quad  \quad 
D(\xx)=\left(\begin{array}{ccc}
	D^{(1)}\left(\xx^{(1)}\right) & \\
	& \ddots & \\
	& & D^{(L)}\left(\xx^{(L)}\right)
\end{array}\right),
$$
with $\mu^{(l)}: \R^{K_{l}} \rightarrow \R^{K_{l}}$ and $D^{(l)}: \R^{K_{l}} \rightarrow \R^{K_{l} \times K_{l}}$, for all $l = 1, \ldots, L$. 
The drift vector $\mu$ drives mutations, which in our case are assumed to be parent-independent and independent of other loci, with the infinitesimal rate of mutation from type $i$ to $j$ at locus $l$ equal to $\alpha_{ij}^{(l)} = \alpha_{j}^{(l)}/2$. The drift component for type $i$ type at locus $l$ thus is
$$
\mu_{i}^{(l)}(\xx^{(l)}) = \sum_{j=1}^{K_l} (\alpha_{ji}^{(l)} x_j^{(l)} - \alpha_{ij}^{(l)} x_i^{(l)}) = \dfrac{1}{2} ( \alpha_{i}^{(l)} - |\aa^{(l)}| x_i^{(l)} ),
$$
where $|\aa^{(l)}| =\sum_{i=1}^{K_{l}} \alpha_{i}^{(l)}$ and $\aa^{(l)}=(\alpha_{1}^{(l)},\ldots,\alpha_{K_{l}}^{(l)})$.
The diffusion terms are of the usual form for WF-type models, namely
$$
d_{i j}^{(l)}(\xx^{(l)}) = x_{i}^{(l)}(\delta_{ij}-x_{j}^{(l)}), \quad\quad 
\delta_{ij} = \begin{cases}
	1, \quad \text {if } i=j \\
	0, \quad \text {if } i \neq j,
\end{cases}
$$
The coupling of the $L$ WF diffusions involved in the model is controlled by the \emph{selection} drift term $D \ \nabla V_{\ss, \JJ}$, where $V_{\ss, \JJ}: [0,1]^{K}  \rightarrow \R$ is defined as 
\begin{equation}\label{V}
	V_{\ss, \JJ}(\xx) = \xx^{T} \ss + \frac{1}{2} {\xx}^T \JJ \xx.
\end{equation}
Here $\ss \in \R_{+}^{K}$ is the vector of single-locus selection parameters, whereas $\JJ \in \R_{+}^{K\times K}$ is a symmetric block matrix containing  only two-locus selection parameters, that is $\JJ^{(lr)} = \left( \JJ^{(rl)}\right)^T \in \R^{K_l \times K_r}$ for all $l,r =1, \ldots, L$, with null sub-matrices $\JJ^{(ll)} = \textbf{O} \in \R^{K_l \times K_l}$ on the diagonal, where $\textbf{O}$ denotes a matrix of zeros. The interaction term $\eqref{V}$ is also known as \textit{fitness potential}, see \cite{AEK19} for further discussion. Thus \eqref{V} includes selection coefficients that apply both within a single locus, as in classical WF diffusions, and among different loci. In the latter case, these generate a weak interaction among the frequencies at these loci by affecting the respective reproduction rates.
Therefore \eqref{sde} is a system of $L$ equations, each reflecting the dynamics of the frequencies at a single locus, coupled through the drift term $D \ \nabla V_{\ss, \JJ}$. When $\nabla V_{\ss, \JJ} = 0$, there is no interaction among loci and \eqref{sde} reduces to a system of $L$ independent WF diffusions each solving 
$$
\d\XX^{(l)}(t) = \mu^{(l)}(\XX^{(l)}(t)) \d t + D^{(l) 1/2}(\XX^{(l)}(t))\d\textbf{W}^{(l)}(t).$$
Alternatively, we can define the c-WF diffusion through its infinitesimal generator. By definition of $V$, we have $\nabla V_{\ss, \JJ}(\xx) =\ss + \JJ\xx$, which implies that the components $g(\xx) := D(\xx) \nabla V_{\ss, \JJ}(\xx)$ are of the form $$
g_{i}^{(l)}(\xx)=\sum_{k=1}^{K_{l}} d_{i k}^{(l)}(\xx^{(l)}) \tilde{\sigma}_{k}^{(l)}(\xx) \quad \text { with } \quad \tilde{\sigma}_{k}^{(l)}(\xx)= \sigma_{k}^{(l)}+\sum_{\substack{r=1 \\ r \neq l}}^{L} \sum_{h=1}^{K_{r}} J_{kh}^{(l r)} x_{h}^{(r)}.
$$
Then the infinitesimal generator $\mathcal{A}$ of the c-WF diffusion, with  domain $D(\mathcal{A}) = C^2(\Delta)$, is
\begin{equation} \label{generator}
	\mathcal{A}  = \sum_{l=1}^{L} \sum_{i=1}^{K_{l}} \left( \mu_{i}^{(l)} ( \xx^{(l)} ) + g_{i}^{(l)} (\xx) \right) \frac{\partial}{\partial x_{i}^{(l)}} + \frac{1}{2} \sum_{l=1}^{L}\sum_{i, j=1}^{K_{l}} d_{i j}^{(l)} ( \xx^{(l)} ) \frac{\partial^{2}}{\partial x_{i}^{(l)} \partial x_{j}^{(l)}} .
\end{equation}
When $\JJ$ vanishes, \eqref{generator} reduces to the sum of $L$ generators 
\begin{equation}\label{indep WF}
\mathcal{A}_{0}^{(l)} = \sum_{i=1}^{K} \left[ \dfrac{1}{2}\left(\alpha_{i} - |\aa| x_i\right)  + \sum_{k=1}^{K} \sigma_k^{(l)} x_i (\delta_{ik} - x_k)\right]  \dfrac{\partial}{\partial x_i} + \dfrac{1}{2} \sum_{i,j=1}^{K} x_i \left( \delta_{ij} - x_j\right) \dfrac{\partial^2}{\partial x_i \partial x_j}
\end{equation} 
which describe a collection  of $L$ independent WF diffusions with selection.


\subsection{Reversibility}

It is well known that the Dirichlet distribution is the invariant measure of the single-locus WF diffusion with parent-independent mutation, cf.~\cite{EG93}, and a tilted Dirichlet distribution is the invariant measure in presence of single-locus selection, cf.~\cite{BEG00}. 
Denote by 
\begin{equation}\label{dirichlet}
\pi_{\aa}(\d\xx)= \xx^{\aa-1} \d\xx = \prod_{l=1}^{L} \left(\xx^{(l)}\right) ^{\aa^{(l)}-1} \d\xx^{(l)} = \prod_{l=1}^{L} \prod_{i=1}^{K_{l}} \big(x_{i}^{(l)}\big)^{\alpha_{i}^{(l)}-1} \d x_i^{(l)} 
\end{equation}
a product of unnormalised Dirichlet densities, where $\aa=(\aa^{(1)},\ldots,\aa^{(L)})$ and $\aa^{(l)}=(\alpha_{1}^{(l)},\ldots,\alpha_{K_{l}}^{(l)})$ as above.
When inter-locus selection is described by \eqref{V}, it was shown in \cite{AEK19} that the c-WF process has invariant distribution
\begin{equation}\label{stationary}
	p_{\aa,\ss,\JJ}(\d\xx) = Z^{-1} \pi_{\aa}(\d\xx) e^{2 V_{\ss, \JJ}(\xx)},
\end{equation}
where $Z$ denotes the normalising constant.
When $\JJ$ vanishes, we recover a product of stationary distribution of $L$ independent WF process as in \eqref{indep WF}, namely in each locus $l\in L$,
$$
p_{\aa, \ss}(\d\xx) \propto \xx^{\aa - 1} e^{2 \xx^T \ss} \d\xx 
= \prod_{i=1}^{K} x_{i}^{\alpha_{i} - 1} e^{2 x_{i} \sigma_{i}}  \d x_{i},
$$
similar to \cite{BEG00}.

Our first result extends the above stationarity and shows that \eqref{stationary} is also the reversible measure for \eqref{generator}. 

\begin{proposition}\label{prop:reversibility}
    The coupled Wright--Fisher diffusion with generator \eqref{generator} is reversible with respect to \eqref{stationary}.
\end{proposition}
\begin{proof}
We leverage on a Theorem 4.6 of \cite{FS86}, where some symmetry-preserving transformations of measure and infinitesimal generator are provided. Specifically, by expanding the interaction function $V_{\ss, \JJ}\in C^2(\Delta)$, 
$$
	V_{\ss, \JJ}(\xx) = \xx^T\ss + \frac{1}{2} \xx^T \JJ \xx = \sum_{l=1}^{L}  \sum_{i=1}^{K_{l}} \bigg[x^{(l)}_i \sigma^{(l)}_i +\frac{1}{2} x^{(l)}_i \sum_{\substack{r=1 \\ r \neq l}}^{L} \sum_{j=1}^{K_{r}} J^{(lr)}_{ij} x^{(r)}_j \bigg],
$$
it follows that \eqref{generator} can be written as
\begin{equation}\label{expansion A_0+rest}
\mathcal{A} = \mathcal{A}_{0} + \sum_{l=1}^{L} \sum_{i=1}^{K_l} \bigg[ \sum_{k=1}^{K_l} a^{(l)}_{i k} \left( \xx \right) \frac{\partial}{\partial x^{(l)}_{k}} V_{\ss, \JJ}(\xx) \bigg] \frac{\partial}{\partial x^{(l)}_{i} },
\end{equation} 
where
$$	
a^{(l)}_{i k} = x_{i} (\delta_{i k} - x_{k}),
$$
and
$$	
\frac{\partial}{\partial x^{(l)}_k} V_{\ss, \JJ}(\xx) = \sigma^{(l)}_k +\frac{1}{2} \sum_{\substack{r=1 \\ r \neq l}}^{L} \sum_{j=1}^{K_{r}} J^{(lr)}_{kj} x^{(r)}_j,
$$
and $\mathcal{A}_{0}$ is the infinitesimal generator of the multi-locus WF diffusion without selection. The latter is known to be symmetric with respect to \eqref{dirichlet} \citep[cf.][Lemma 4.1]{EK81}, in the sense that $\int f (\mathcal{A}_{0} g ) \d\pi_{\alpha} = \int g (\mathcal{A}_{0} f ) \d\pi_{\alpha}$ for all $ f,g \in D(\mathcal{A}_{0})$. Then, by virtue of \eqref{expansion A_0+rest}, it follows that $\mathcal{A}$ is symmetric with respect to \eqref{stationary}. 
\end{proof}
The reversibility established above will be needed in Section \ref{sec: filtering} to study filtering and smoothing for c-WF diffusions.


\subsection{Duality}\label{sec:duality}
Given two Markov processes $\{X(t):t\ge0\}$ and $\{M(t):t\ge0\}$ on $\S$ and $\S'$ respectively, we say $X(t)$ is dual to $M(t)$ with respect to a bounded measurable function $h:\S \times \S' \rightarrow \R$ if \begin{equation}\label{dual}
	\mathbb{E}\left[ h(X(t),m)\mid X(0)=x \right] = {\mathbb{E}}\left[ h(x,M(t)) \mid M(0)=m\right] \qquad \forall x\in \S, m \in \S', t>0.
\end{equation}
See \cite{EK86,JK14} for reviews, details and applications. \cite{FHK21} identified a multivariate jump process on $\N^{K}$, which we will define below, as dual to a c-WF diffusion with respect to function
\begin{equation} \label{h}
	h(\xx, \mm) = \frac{1}{k(\mm)}  \xx^{\mm} = \frac{1}{k(\mm)} \prod_{l=1}^{L} \prod_{i=1}^{K_{l}} \left( x_{i}^{(l)} \right)^{m_{i}^{(l)}},
\end{equation}
where, for $\tilde{\XX} \sim p_{\aa, \ss, \JJ}(\d\xx)$, 
$$
k(\mm)= \mathbb{E}\bigg[\prod_{l=1}^{L} \prod_{i=1}^{K_{l}}\big(\tilde{X}_{i}^{(l)}\big)^{m_{i}^{(l)}}\bigg].
$$
Notice also that \eqref{h} is of the same form as the duality function for the single-locus case in \cite{EG09,PR14,GJL19}, since this model is labelled, as opposed, e.g., to \cite{Gea24}.

Let $\mathbf{e}_{i}^{(l)}$ be the standard unit vector in $\R^{K}$ in the $i$-th direction on the $l$-th locus.
In the special case with parent-independent mutations, the dual process is a pure-jump Markov process $\MM:=\{\MM(t),t\ge0\}$, with $\MM(t)\in \N^{K}$, which jumps from $\mm \in \N^{K} \backslash\set{\textbf{0}}$ to 
	\begin{itemize}
		\item $\mm-\mathbf{e}_{i}^{(l)}$ (\emph{coalescence}) for $i=1, \ldots K_{l},\ l=1, \ldots, L$, s.t. $m_{i}^{(l)} \geq 2$, at rate $$
		q\left(\mm, \mm-\mathbf{e}_{i}^{(l)}\right) = \dfrac{m_{i}^{(l)} (m_{i}^{(l)}-1)}{2} \dfrac{k(\mm-\mathbf{e}_{i}^{(l)})}{k(\mm)};
		$$
		\item $\mm-\mathbf{e}_{i}^{(l)}+\mathbf{e}_{j}^{(l)}$ (\emph{mutation}) for $i, j=1, \ldots K_{l},\ l=1, \ldots, L$, s.t. $i \neq j, m_{i}^{(l)} \geq 1$, at rate $$
		q\left(\mm, \mm-\mathbf{e}_{i}^{(l)} + \mathbf{e}_{j}^{(l)}\right) = m_{i}^{(l)} \dfrac{\alpha_{i}^{(l)}}{2} \dfrac{k(\mm-\mathbf{e}_{i}^{(l)} + \mathbf{e}_{j}^{(l)})}{k(\mm)};
		$$
		\item $\mm+\mathbf{e}_{j}^{(l)}$ (\emph{branching}) for $j=1, \ldots K_{l},\ l=1, \ldots, L$, at rate 
		\begin{equation}\label{branching rate}
			q\left(\mm, \mm + \mathbf{e}_{j}^{(l)} \right) = \Bigl(|\mm^{(l)}| \sum_{\substack{k=1 \\ k \neq j}}^{K_{l}} \sigma_{k}^{(l)} + \sum_{\substack{r=1 \\ r \neq l}}^{L} \sum_{i=1}^{K_{r}} m_{i}^{(r)} J_{j i}^{(l r)}\Bigr) \dfrac{k(\mm + \mathbf{e}_{j}^{(l)})}{k(\mm)};
\end{equation} 
		\item $\mm+\mathbf{e}_{j}^{(l)}+\mathbf{e}_{h}^{(r)}$ (\emph{double branching}) for $j=1 \ldots K_{l},\ h=1 \ldots, K_{r},\ l, r=1, \ldots, L,\ r>l$, at rate $$
			q\left(\mm, \mm + \mathbf{e}_{j}^{(l)} + \mathbf{e}_{h}^{(r)}\right) = \Bigl[\left( |\mm^{(l)}| + |\mm^{(r)}|\right) \sum_{\substack{k=1 \\ k \neq j}}^{K_{l}} \sum_{\substack{n=1 \\ n \neq h}}^{K_{r}} J_{k n}^{(l r)}\Bigr] \dfrac{k(\mm + \mathbf{e}_{j}^{(l)} + \mathbf{e}_{h}^{(r)})}{k(\mm)},
		$$
	\end{itemize}
	where $|\mm^{(l)}| =\sum_{i=1}^{K_{l}} m_{i}^{(l)}$. 
See \cite{FHK21} for further details and interpretation.

\section{Filtering coupled Wright--Fisher diffusions}\label{sec: filtering}

We are going to cast c-WF diffusions in a hidden Markov model statistical framework, whose general structure will be recalled shortly, and derive, under certain assumptions on the data collection, the conditional distributions of the diffusion states under different types of conditioning: on past data, leading to the signal \emph{predictive} distribution; on past and present data, leading to the \emph{filtering} distribution; and on the entire dataset, that is including observation that become available at a later time, leading to \emph{marginal smoothing} distributions. 


A hidden Markov model (HMM) is defined by a two-component process $\set{\left(X_{n},Y_{n}\right) : n\ge0}$, where $X_{n}$ is an unobserved Markov chain, often called \emph{latent signal}, while $Y_{n}$ denotes the  observation (or the vector of observations) collected at step $n$. Such a pair is a HMM if the distribution of $Y_{n}$ only depends on the value of $X_{n}$, plus possibly additional global parameters, making $Y_{n}$ conditionally independent of $X_{k},Y_{k}, k\ne n$, given $X_{n}$. See \cite{CMR05} for a general treatment of HMMs. 

Here we consider the signal to be the discrete-time skeleton of a c-WF diffusion, that is, if $\XX$ has generator \eqref{generator}, we let $X_{n}:=\XX(t_{n})$ for $0=t_{0}<t_{1}<\cdots<T$. Note that the intervals between data collection times need not be equal. 
We assume the following specification for the HMM:
\begin{equation}\label{HMM}
\begin{aligned}
\XX \sim&\, \text{c-WF}_{\aa, \ss, \JJ}, \quad \XX(0)\sim p_{\aa, \ss, \JJ}, \\
\YY(t_{n}) | \XX(t_{n}) = \xx &\, \sim 
\text{MN}_{L}(\NN_{n}, \xx)
\end{aligned}
\end{equation}
where $\XX \sim\text{c-WF}_{\aa, \ss, \JJ}$ denotes that $\XX$ is a c-WF diffusion with operator \eqref{generator}, $p_{\aa, \ss, \JJ}$ is the initial distribution chosen as in \eqref{stationary}, 
and
\begin{equation}\nonumber
\text{MN}_{L}(\NN_{n}, \xx)= \prod_{l=1}^{L} \text{MN}(N_n^{(l)}, \xx^{(l)})
\end{equation} 
where
 $\text{MN}(N_n^{(l)}, \xx^{(l)})$ denotes a Multinomial pdf with size $N_n^{(l)}$ and categorical probabilities $\xx^{(l)}=(\xx^{(l)}_{1},\ldots,\xx^{(l)}_{K})$. Here the values in  $\NN_{n}=(N_n^{(1)},\ldots,N_n^{(L)})$ are assumed to be given for each observation time $t_{n}$.

Note now that \eqref{stationary} and \eqref{h} imply 
\begin{equation}\label{conjugacy}
\xx^{\nn} p_{\aa, \ss, \JJ}(\mathrm{d}\xx) \propto h(\xx, \nn) p_{\aa, \ss, \JJ}(\mathrm{d}\xx) = p_{\aa+\nn, \ss, \JJ}(\mathrm{d}\xx),
\end{equation}
for all $\xx \in \Delta$ and $\nn \in \N^{K} \setminus \set{0}$, where
$p_{\aa + \nn, \ss, \JJ}$ defined as in \eqref{stationary} with $\aa$ replaced by $ \aa+\nn$ where $\nn=(\nn^{(1)},\ldots,\nn^{(L)})\in \N^{K} \setminus \set{0}$ and $\nn^{(l)}=(n_{1}^{(l)},\ldots,n_{K_{l}}^{(l)})\in \N^{K_{l}} \setminus \set{0}$.
Then, under \eqref{HMM}, the distribution of $ \XX(t_{0})|\YY(t_{0})$ is $p_{\aa+\nn, \ss, \JJ}(\mathrm{d}\xx)$ with $\nn$ being the vector of multiplicities observed in $\YY(t_{0})$. In general we are interested in the signal predictive distribution, i.e., the law of $\XX(t_{n+1})|\YY(t_{0}),\ldots,\YY(t_{n})$, and in the filtering distribution, i.e., the law of $\XX(t_{n+1})|\YY(t_{0}),\ldots,\YY(t_{n+1})$. These are no longer single kernels as in \eqref{conjugacy}, but can be computed recursively as described  by the next result. 

\begin{proposition}\label{prop:filtering}
Let $(\XX(t_{n}),\YY(t_{n}))_{n\ge0}$ be as in \eqref{HMM}, and let 
\begin{equation}\label{porp_filtering_assumption}
\XX(t_{n})|\YY(t_{0}),\ldots,\YY(t_{n}) \sim \sum_{\mm>\oo} \hat{c}_{n}(\mm)p_{\aa + \mm, \ss, \JJ}(\d\xx), 
\end{equation} 
where $\sum_{\mm>\oo} \hat{c}_{n}(\mm)=1$.
Then 
\begin{equation}\label{signal prediction}
    \XX(t_{n+1})|\YY(t_{0}),\ldots,\YY(t_{n})\sim 
    \sum_{\nn>\oo} 
    c_{n+1}(\nn) p_{\aa + \nn, \ss, \JJ},
\end{equation} 
where $c_{n+1}(\nn)\propto \sum_{\mm>\oo} \hat{c}_{n}(\mm)p_{\mm, \nn}(t_{n+1}-t_{n})$ and $p_{\mm, \nn}(t)$ are the transition probabilities of $\textbf{M}$ as in Section \ref{sec:duality}. 
Furthermore, let $\YY(t_{n+1})$ be the data collected at time $t_{n+1}$, with $\yy\in \N^{K} \setminus \set{0}$ describing the observed multiplicities of types, and let $\text{MN}_{L}(\yy;\NN_{n+1}, \xx)$ be as in \eqref{HMM}, evaluated at $\yy$. Then 
\begin{equation}\nonumber
    \XX(t_{n+1})|\YY(t_{0}),\ldots,\YY(t_{n+1}) \sim 
    \sum_{\nn>\oo}
    \hat{c}_{n+1}(\nn) p_{\aa + \nn+\yy, \ss, \JJ},
\end{equation} 
where $\hat{c}_{n+1}(\nn)\propto c_{n+1}(\nn)d_{n+1}(\mm,\yy)$ and $d_{n+1}(\mm, \yy) := \int_{\Delta} \mathrm{MN}_{L}(\yy;\NN_{n+1}, \xx) p_{\aa + \mm, \ss, \JJ}(\d \xx)$, with $\NN_{n+1}$ the vector of sample sizes collected at time $t_{n+1}$.
\end{proposition}

\begin{proof}
By virtue of Proposition \ref{prop:reversibility}, the c-WF diffusion is reversible with respect to \eqref{stationary}. Furthermore, the duality identity \eqref{dual} holds with respect to functions in \eqref{h}. These yield conditions analogous to Assumptions 1 and 3 in Proposition 2.1 in \cite{KKK24}, whose implications are spelled out in the rest of the proof. 
In particular, from these it follows that we can exchange the expectation with respect to the law of $\XX(t)$ with the expectation with respect to the law of $\MM(t)$, and determine the signal predictive distribution for $\XX(t_{n+1})|\YY(t_{0}),\ldots,\YY(t_{n})$ using the dual process transition probabilities, giving rise to a countable mixture of distributions $p_{\aa + \nn, \ss, \JJ}$ with mixture weights $c_{n+1}(\nn) \propto \sum_{\mm>\oo} \hat{c}_{n}(\mm)p_{\mm, \nn}(t_{n+1}-t_{n})$. 
Indeed, by denoting with $P_{t}(\xx, \d\xx')$ the transition semigroup of the c-WF (where the notation $(\xx,\xx')$ is equivalent to $(\xx'|\xx)$), we see that the law of $\XX(t_{n+1})|\YY(t_{0}),\ldots,\YY(t_{n})$ equals
\begin{align}
\int_{\Delta} &\left[ \sum_{\mm>\oo} \hat{c}_{n}(\mm)p_{\aa + \mm, \ss, \JJ}(\d\xx) \right] P_{t_{n+1}-t_{n}}\left(\xx, \d \xx'\right) \\
	& = \sum_{\mm>\oo} \hat{c}_{n}(\mm) \int_{\Delta} h(\xx, \mm) p_{\aa, \ss, \JJ}(\d \xx) P_{t_{n+1}-t_{n}}\left(\xx, \d \xx'\right) .
\end{align}
Proposition \ref{prop:reversibility} now implies the detailed balance condition, which here reads
\begin{equation}\nonumber
p_{\aa, \ss, \JJ}(\d \xx) P_{t_{n+1}-t_{n}}(\xx, \d \xx') 
=p_{\aa, \ss, \JJ}(\d \xx') P_{t_{n+1}-t_{n}}(\xx', \d \xx),
\end{equation} 
from which the expression in the preceding display is equal to
\begin{equation}\nonumber
\begin{aligned}
\sum_{\mm>\oo} &\hat{c}_{n}(\mm) \int_{\Delta}  h(\xx, \mm) p_{\aa, \ss, \JJ}\left(\d \xx'\right) P_{t_{n+1}-t_{n}}\left(\xx', \d \xx\right) \\
    & = \sum_{\mm>\oo} \hat{c}_{n}(\mm) p_{\aa, \ss, \JJ}\left(\d \xx'\right) \int_{\Delta}  h(\xx, \mm) P_{t_{n+1}-t_{n}}\left(\xx', \d \xx\right) \\
	& = \sum_{\mm>\oo} \hat{c}_{n}(\mm) p_{\aa, \ss, \JJ}\left(\d \xx'\right) \mathbb{E}\left[ h(\XX(t_{n+1}),\mm)\mid \XX(t_{n})=\xx' \right].
\end{aligned}
\end{equation} 
Then, from \eqref{dual} with \eqref{h}, the previous equals
\begin{equation}\nonumber
\begin{aligned}
 \sum_{\mm>\oo} &\hat{c}_{n}(\mm) p_{\aa, \ss, \JJ}\left(\d \xx'\right) 
    \mathbb{E}\left[ h(\xx',\MM(t_{n+1}))\mid \MM(t_{n})=\mm \right].
\end{aligned}
\end{equation} 
Writing the expectation as an expansion based on the dual transition probabilities, and rearranging the terms, leads now to
\begin{equation}\nonumber
\begin{aligned}
\sum_{\mm>\oo} &\hat{c}_{n}(\mm) p_{\aa, \ss, \JJ}\left(\d \xx'\right) \left[ \sum_{\nn > \oo} p_{\mm, \nn}(t_{n+1}-t_{n}) h(\xx', \nn) \right]\\
    & = \sum_{\mm>\oo} \sum_{\nn > \oo} \hat{c}_{n}(\mm) p_{\mm, \nn}(t_{n+1}-t_{n}) h(\xx', \nn) p_{\aa, \ss, \JJ}\left(\d \xx'\right) \\
    & = \sum_{\mm>\oo} \sum_{\nn > \oo} \hat{c}_{n}(\mm) p_{\mm, \nn}(t_{n+1}-t_{n})p_{\aa + \nn, \ss, \JJ}\left(\d \xx'\right) \\
    & = \sum_{\nn > \oo} \sum_{\mm>\oo} \hat{c}_{n}(\mm) p_{\mm, \nn}(t_{n+1}-t_{n})p_{\aa + \nn, \ss, \JJ}\left(\d \xx'\right) \\
    & = \sum_{\nn > \oo} c_{n+1}(\nn) p_{\aa + \nn, \ss, \JJ}(\d \xx') ,
\end{aligned}
\end{equation} 
with $c_{n+1}(\nn)\propto \sum_{\mm>\oo} \hat{c}_{n}(\mm)p_{\mm, \nn}(t_{n+1}-t_{n})$, where we have also used \eqref{conjugacy}.

Clearly, \eqref{stationary} also implies $\xx^{\yy}  p_{\aa+\nn, \ss, \JJ}(\mathrm{d}\xx) \propto p_{\aa+\nn+\yy, \ss, \JJ}(\mathrm{d}\xx)$. Furthermore, notice that the posterior distribution of a mixture is a mixture of posteriors whose weights are also conditioned on the observations, namely
\begin{equation}\nonumber
Z\sim \sum_{i}q(i)p_{i}(z)
\quad \Rightarrow \quad 
Z|Y=y\sim \sum_{i}\frac{q(i)p(y|i)}{\sum_{j}q(j)p(y|j)}p_{i}(z|y)
\end{equation} 
given the first is the same as $Z|i\sim p_{i}$ with probability $q(i)$, hence $Z|i,y\sim p_{i}(z|y)$ with probability $q(i|y)$, which is proportional to $q(i)p(y|i)$ by Bayes's theorem. Hence, we conclude that the filtering distribution for $\XX(t_{n+1})|\YY(t_{0}),\ldots,\YY(t_{n+1})$ is a countable mixture of distributions $p_{\aa + \nn+\yy, \ss, \JJ}$ with mixture weights $\hat{c}_{n+1}(\nn) \propto  c_{n+1}(\nn) d_{n}(\mm, \yy)$ as given in the statement.

\end{proof}

Hence, under \eqref{HMM}, the filtering and signal predictive distributions are all identified as countable mixtures of kernels in the same family as \eqref{stationary}. These findings fall in the family of filtering distributions described in \cite{CG06} (cf.~eq.(2)), which evolve within the set of countable  mixtures of specified parametric kernels.
Note however that in the present setting the specific form of the semigroup of the c-WF is unknown, and therefore the sufficient conditions there specified cannot be checked. The same conclusions can nonetheless be drawn here using a duality-based approach. 

The mixture weights in Proposition \ref{prop:filtering} depend on the dual process transitional probabilities. We can therefore associate  each mixture component to a node of $L$ coupled $K_{l}$-dimensional directed graphs, as illustrated in Figure \ref{fig: loci}. With $L=2$ loci and $K_1=K_2=2$ types at each locus, the $4$-dimensional dual process $\MM$ can be though of as a joint jump process on paired 2-dimensional graphs. The potential jumps are indicated by arrows in the graphs, and are determined by the transition rates in Section \ref{sec:duality}. For example, a coalescence event determines a downward jump to a neighbour node in direction $i$. Due to the dependence among the $L$ components of $\XX$, the dynamics on each graph are not independent, and the two graphs could be seen as single $4$-dimensional graph. For example, the red nodes in Figure \ref{fig: loci} corresponds to the state $\mm = (2,0,1,1)$ for the dual, and the probability of jumping upward in each graph depends on both states (cf., e.g., \eqref{branching rate}).  
\begin{figure}[ht!]
	\centerline{
	\iffigures\includegraphics[width=.9\textwidth]{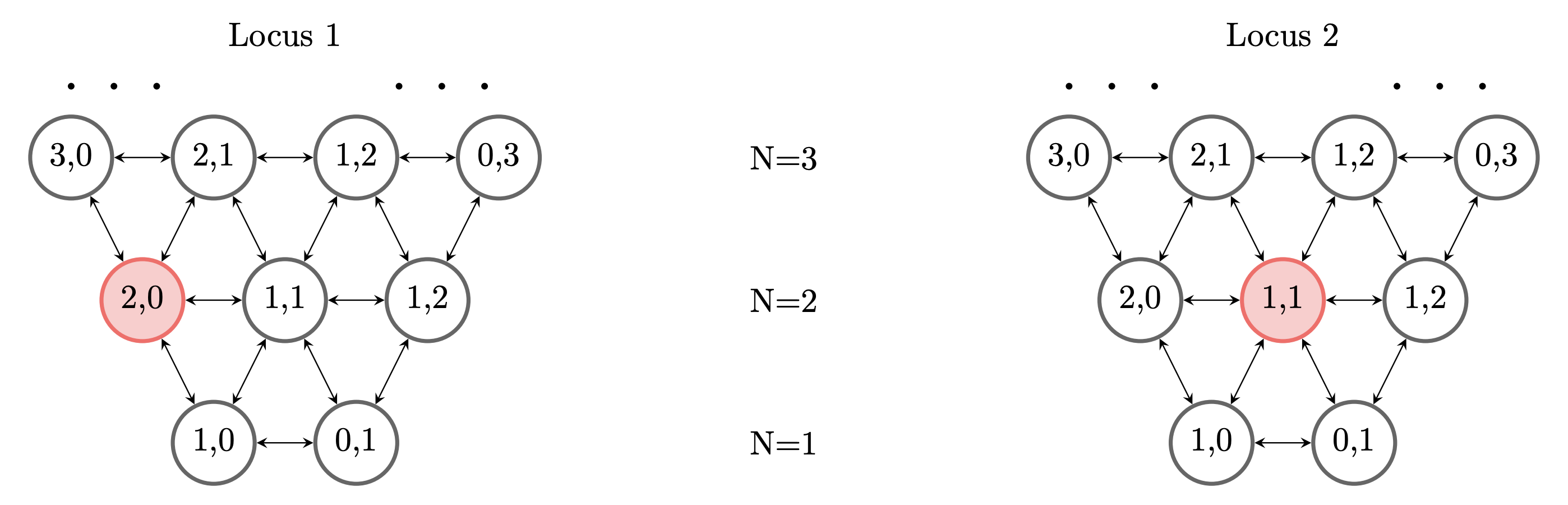}
	\else 
	\text{\red{decomment \emph{figurestrue} in preamble}} \fi
	}
	\caption{Graphical representation of the nodes identifying the states of the dual and the mixture components of the filtering distributions.}\label{fig: loci}
\end{figure}

Note that, at least in principle, one could extend the strategy laid out in \cite{BEG00} for exploiting duality in order to identify an expansion for the diffusion transition function. This goal however entails a non trivial adaptation of their methodology and is left for future work.

Next we turn to the marginal smoothing problem. When further observations become available at a later time, sometimes it is useful to improve previous estimates with the additional data. This procedure typically results in smoother estimates, whence the name.
Determining the smoothing distribution at time $t_{i}$ requires in general integrating out the signal trajectory at all other times, conditional on the entire dataset. This calculation is typically unfeasible, but duality provides an efficient way of identifying the quantities of interest.

\begin{proposition}\label{prop:smoothing}
For $0 \le i \le n-1$, the marginal smoothing density of $\XX(t_{i})$ given $\YY(t_{0}),\ldots,\YY(t_{n})$ admits representation
\begin{equation}
    \XX(t_{i})|\YY(t_{0}),\ldots,\YY(t_{n}) \sim     
    \sum_{\mm,\nn>\oo}
    w_{i}(\mm,\nn) p_{\aa+\mm+\nn, \ss, \JJ}(\d\xx),
\end{equation}
where, for $\hat{c}_{i}(\nn)$ as in Proposition \ref{prop:filtering}, 
\begin{align}
    & w_{i}(\mm,\nn) \propto C_{\mm, \nn} \omega_{i+1}(\mm) \hat{c}_{i}(\nn), \quad \quad C_{\mm,\nn}= \frac{k(\mm)k(\nn)}{k(\mm+\nn)},\\
    & \omega_{i+1}(\mm) = \sum_{\nn >\oo} \omega_{i+2}(\nn) d_{i+1}(\nn, \yy(t_{i+1}))p_{\nn + \yy(t_{i+1}), \mm}(t_{i+1}-t_{i}),
\end{align}
with $d_{i+1}(\cdot,\cdot)$ and $p_{\cdot,\cdot}(t)$ as in Proposition \ref{prop:filtering}.
\end{proposition}
\begin{proof}
Bayes' Theorem and conditional independence allow to write, for every $0 \le i \le n-1$,
\begin{equation}\nonumber
    p(\XX(t_{i})|\YY(t_{0}),\ldots,\YY(t_{n}) )
    \propto p( \YY(t_{i+1}),\ldots,\YY(t_{n}) | \XX(t_{i}) )p( \XX(t_{i})|\YY(t_{0}),\ldots,\YY(t_{i}))
\end{equation}
where the last factor is given in Proposition \ref{prop:filtering}. It can be now shown by induction that
\begin{equation}\label{cost-to-go}
    \YY(t_{i+1}),\ldots,\YY(t_{n}) | \XX(t_{i}) \sim \sum_{\mm > \oo} \omega_{i+1}(\mm) h(\xx, \mm).
\end{equation}
Indeed, \eqref{cost-to-go} holds for $i=n-1$ since \eqref{conjugacy} implies
\begin{equation}\nonumber
   \text{MN}_{L}(\yy; \NN_{n}, \xx) =  h(\xx, \yy) d_{n}(\oo, \yy),
\end{equation}
and $\mathbb{E}\left[ h(\XX(t_{n}),\mm)\mid \XX(t_{n-1})=\xx' \right]$ can be written in terms of the dual transition probabilities, as done in the proof of Proposition \ref{prop:filtering}. Assume then \eqref{cost-to-go} holds for $i+1$, and note that $ \YY(t_{i+1}),\ldots,$ $\YY(t_{n}) | \XX(t_{i}) $ has law
\begin{equation}\nonumber
\begin{aligned}
 \int_{\Delta}
    & \text{MN}_{L}(\yy(t_{i+1}); \NN_{i+1}, \xx(t_{i+1})) p(\yy(t_{i+2}),\ldots,\yy(t_{n})| \xx(t_{i+1}))P_{t_{i+1}-t_{i}} \left(\xx(t_{i}), \d \xx(t_{i+1}) \right).
\end{aligned}
\end{equation}
Since now $h(\xx, \mm) \text{MN}_{L}(\yy; \NN_{n}, \xx) =  h(\xx, \mm + \yy) d_{n}(\mm, \yy)$
holds for all $\xx \in \Delta$ and $\mm \in \N^{K}\setminus \set{0}$, \eqref{cost-to-go} follows from the expansion of the conditional expectation with respect to the dual transition probabilities, upon rearranging the terms. 
Furthermore, $h$ in \eqref{h} satisfies
\begin{equation}
    h(\xx, \mm) h(\xx, \nn) = \frac{k(\mm)k(\nn)}{k(\mm+\nn)} h(\xx, \mm+\nn),
\end{equation}
for all $\xx \in \Delta$, $\mm,\nn \in \N^{K}\setminus \set{0}$. Hence the law of $ \XX(t_{i})|\YY(t_{0}),\ldots,\YY(t_{n}) $ is proportional to
\begin{equation}\nonumber
\begin{aligned}
 \sum_{\mm > \oo}& \omega_{i+1}(\mm) h(\xx, \mm) \sum_{\nn > \oo} \hat{c}_{i}(\nn) h(\xx, \nn) p_{\aa, \ss, \JJ}(\d \xx) \\
    & = \sum_{\mm > \oo} \sum_{\nn > \oo} \omega_{i+1}(\mm) \hat{c}_{i}(\nn) C_{\mm, \nn} h(\xx, \mm + \nn) p_{\aa, \ss, \JJ}(\d \xx) \\
    & = \sum_{\mm > \oo} \sum_{\nn > \oo} \omega_{i+1}(\mm) \hat{c}_{i}(\nn) C_{\mm, \nn} p_{\aa + \mm + \nn, \ss, \JJ}(\d \xx),
\end{aligned}
\end{equation}
which, upon normalization, leads to the claim.
\end{proof}

\section{Implementation and illustration}\label{sec:illustration}

Propositions \ref{prop:filtering} and \ref{prop:smoothing} provide recursion for reweighing the mixture components in order to evaluate the filtering and marginal smoothing distributions given the available data. Algorithm \ref{alg1} lays out the pseudocode for calculating the filtering mixtures as in Proposition \ref{prop:filtering}, while Algorithm \ref{alg2} does the same for the smoothing mixtures as in Proposition \ref{prop:smoothing}.

\begin{algorithm}[ht!]
\footnotesize
\vspace{2mm}
\hspace{-5mm} \KwData{$\yy(j) = (\yy^{(1)}(t_{j}), \ldots, \yy^{(L)}(t_{j}))$ and $t_j$ for $j=0,1,\ldots n$}
\hspace{-5mm} Set parameters $\aa, \ss$ and $\JJ$ for $p_{\aa, \ss, \JJ}(\d\xx)$\\[0mm]
\SetKwBlock{Begin}{Initialise}{end}
\Begin{
    \SetKwBlock{Begin}{For $\mm > \oo$}{end}
    \Begin{
        Compute $c_0(\mm) = p_{\yy(t_0), \mm}(t_1-t_{0})$
    }
}

\SetKwBlock{Begin}{For $j=1,\ldots,n-1$}{end}
\Begin{
    \SetKwBlock{Begin}{Update}{end}
    \Begin{
        Get $\yy(t_{j})$\\
        \SetKwBlock{Begin}{For $\mm > \oo$}{end}
        \Begin{
            Compute $d_{j}(\mm, \yy(t_{j}))$\\
            $a(\mm) \leftarrow c_{j-1}(\mm) \cdot d_{j}(\mm, \yy(t_{j}))$
        }
        \SetKwBlock{Begin}{For $\mm > \oo$}{end}
        \Begin{
            $\hat{c}_{j}(\mm) \leftarrow \frac{a(\mm)}{\sum_{\mm > \oo} a(\mm)}$
        }
    }
    
    \SetKwBlock{Begin}{Prediction}{end}
    \Begin{
        \SetKwBlock{Begin}{For $\nn > \oo$}{end}
        \Begin{
            $c_j(\nn) \leftarrow \sum_{\mm > \oo}  \hat{c}_{j}(\mm) \cdot p_{\mm + \yy(t_{j}), \nn}( t_{j+1}-t_j)$
        }
    }
}

\SetKwBlock{Begin}{Last update}{end}
\Begin{
    Get $\yy(t_{n})$\\
    \SetKwBlock{Begin}{For $\mm > \oo$}{end}
    \Begin{
        Compute $d_{n}(\mm, \yy(t_{n}))$\\
        $a(\mm) \leftarrow c_{n-1}(\mm) \cdot d_{n}(\mm, \yy(t_{n}))$
    }
    \SetKwBlock{Begin}{For $\mm > \oo$}{end}
    \Begin{
        $\hat{c}_{n}(\mm) \leftarrow \frac{a(\mm)}{\sum_{\mm > \oo} a(\mm)}$
    }
}

\textbf{Return} $\hat{c}_{n}(\mm)$
\caption{Evaluation of the mixture weights $\hat{c}_{n}(\mm)$ in Proposition \ref{prop:filtering}} \label{alg1}
\end{algorithm}

\begin{algorithm}[ht!]
\footnotesize
\vspace{2mm}
\hspace{-5mm} \KwData{$\yy(j) = (\yy^{(1)}(t_{j}), \ldots, \yy^{(L)}(t_{j}))$ and $t_j$ for $j=0,1,\ldots,n$}
\hspace{-5mm} Set parameters $\aa, \ss$ and $\JJ$ for $p_{\aa, \ss, \JJ}(\d\xx)$\\[0mm]
\SetKwBlock{Begin}{Initialise}{end}
\Begin{
    \SetKwBlock{Begin}{For $\mm > \oo$}{end}
    \Begin{
        Obtain $\hat{c}_{j}(\mm)$ for $j=0,1,\ldots n$ from Algorithm \ref{alg1}\\
        Compute $c^{*}_{n-1}(\mm) = p_{\yy(n), \mm}(t_{n+1}-t_{n})$
    }
}

\SetKwBlock{Begin}{For $j=n-1,\ldots,1$}{end}
\Begin{
    \SetKwBlock{Begin}{Update}{end}
    \Begin{
        Get $\yy(t_{j})$\\
        \SetKwBlock{Begin}{For $\mm > \oo$}{end}
        \Begin{
            Compute $d_{j}(\mm, \yy(t_{j}))$\\
            $\omega_{j+1}(\mm) \leftarrow c^{*}_{j}(\mm) \cdot d_{j}(\mm, \yy(t_{j}))$
        }
    }
    
    \SetKwBlock{Begin}{Prediction}{end}
    \Begin{
        \SetKwBlock{Begin}{For $\nn > \oo$}{end}
        \Begin{
            $c^{*}_{j-1}(\nn) \leftarrow \sum_{\mm > \oo}  \hat{c}_{j}(\mm) \cdot p_{\mm + \yy(j-1), \nn}( t_{j}-t_{j-1})$
        }
    }
}

\SetKwBlock{Begin}{Last update}{end}
\Begin{
        Get $\yy(0)$\\
        \SetKwBlock{Begin}{For $\mm > \oo$}{end}
        \Begin{
            Compute $d_{0}(\mm, \yy(t_{0}))$\\
            $\omega_{1}(\mm) \leftarrow c^{*}_{0}(\mm) \cdot d_{0}(\mm, \yy(t_{0}))$
        }
    }

\SetKwBlock{Begin}{For $j=0, \ldots, n-1$}{end}
\Begin{
    \SetKwBlock{Begin}{For $\mm, \nn > \oo$}{end}
    \Begin{
        $a(\mm, \nn) \leftarrow C_{\mm, \nn} \omega_{j+1}(\nn) \hat{c}_{j}(\mm)$
    }
    \SetKwBlock{Begin}{For $\mm, \nn > \oo$}{end}
    \Begin{
        $w_{j}(\mm,\nn) \leftarrow \frac{a(\mm, \nn)}{\sum_{\mm,\nn > \oo} a(\mm, \nn)}$
    }
}

\textbf{Return} $w_{j}(\mm,\nn)$ for $j=0, \ldots, n-1$
\caption{Evaluation of the mixture weights $w_{i}(\mm,\nn)$ in Proposition \ref{prop:smoothing}} \label{alg2}
\end{algorithm}

The concrete implementation of these algorithm must however deal with two key hurdles. The first is the fact that the normalising constant in \eqref{stationary} and the constant $k(\mm)$ in \eqref{h} involve integrals which are not available in closed form. \cite{AEK19} and \cite{FHK21} derive a series representation of the normalising constant $Z$ when $\ss = 0$ by involving the Kummer (confluent hypergeometric) function. In general, these integrals need to be approximated with appropriate methods. This hurdle carries over to the evaluation of the marginal distributions $d_{n}(\mm, \yy)$ (cf.~Propositions \ref{prop:filtering}-\ref{prop:smoothing}), as one can check that 
\begin{equation}
    d_{n}(\mm, \yy)  = \prod_{l=1}^L \left( \dfrac{N^{(l)}_{n}!}{\prod_{i=1}^{K_l} y_i^{(l)}!} \right) \frac{\tilde C(\mm+\yy)}{\tilde C(\mm)}
\end{equation}
where
\begin{equation}
    \tilde C(\mm) := \int_{\Delta} \xx^{\mm + \aa - 1} e^{2  V(\xx; \ss, \JJ)} \d\xx
\end{equation}
is in general not available.
Here we adopt a naive Monte Carlo sampler, as Monte Carlo approximations typically outperform quadrature methods in high dimensions, particularly in terms of computational efficiency and ease of implementation, at the cost of higher variability in accuracy. Concretely, $Z$ in \eqref{stationary} is approximated by drawing $B$ \emph{iid} samples $\xx^{(b)}$ from $\pi_{\aa}$ and computing $B^{-1}\sum_{b=1}^{B}e^{2 V_{\ss, \JJ}(\xx^{(b)})}$. The other needed functionals can be computed similarly. 

The second difficulty arises in handling the dual process $\MM$, as no closed formula for its transition probabilities is available. 
Here we adopt the strategy used in \cite{KKK24}, whereby the transition probabilities are approximated by means of several simulated trajectories of the dual. More precisely, one runs multiple times the jump process starting in $\mm$ over the needed time interval, and the transition probability $p_{\mm,\nn}(\Delta t)$ is then approximated by the empirical distribution over the arrival points $\nn$. This procedure is repeated for every starting point $\mm$ with positive mass in the last available mixture, and then the probability masses over the arrival points are aggregated appropriately. The simulation of each run of the dual process is conducted using the Gillespie algorithm \citep{G01}, which alternates drawing exponential holding times and jumps of the embedded chain until the time interval is exhausted. This strategy, besides being easy to implement, has the additional advantage of automatically reducing the cardinality of the transition probability support, which is countable, to a finite set. This corresponds to the support of the empirical distribution obtained as explained above, where the Monte Carlo runs of the dual process place probability mass in the most important part of the state space, in accordance to the dual process dynamics. 

Additionally, we adopt a pruning strategy, which was discussed in \cite{KKK21}. This consists in retaining only the points of the above simulation which are assigned probability mass above a desired threshold, or the points with decreasingly ranked probability mass which collectively provide mass over a desired threshold. Its application is straightforward and, despite not being necessary, makes the implementation more efficient. 



To illustrate, we compute the filtering distributions for a two-locus two-allele setting, i.e., when $L=K_{1}=K_{2}=2$, as in Figure \ref{fig: loci}. We assume pairwise interactions occur according to the rate matrix
\begin{equation}
\JJ = 
\begin{pmatrix}
  0 & 0 & J_{1} & 0 \\
  0 & 0 & 0 & J_{2} \\
  J_{1} & 0 & 0 & 0 \\
  0 & J_{2} & 0 & 0 \\
\end{pmatrix}
\end{equation}
and within-locus selection is determined by $\ss$.
We simulate a c-WF trajectory with $\aa=(1.8,1.4,1.9,$ $1.7)$, $\ss=(0.5,0,0,1.2)$, $J_1=0.9$, and $J_2=1.8$, by using an approximating Markov chain as in \cite{AEK19}, leveraging on their scaling limit.

\begin{figure}[h!]
    \centering
    \begin{subfigure}[t]{0.28\textwidth}
        \iffigures\includegraphics[width=\linewidth]{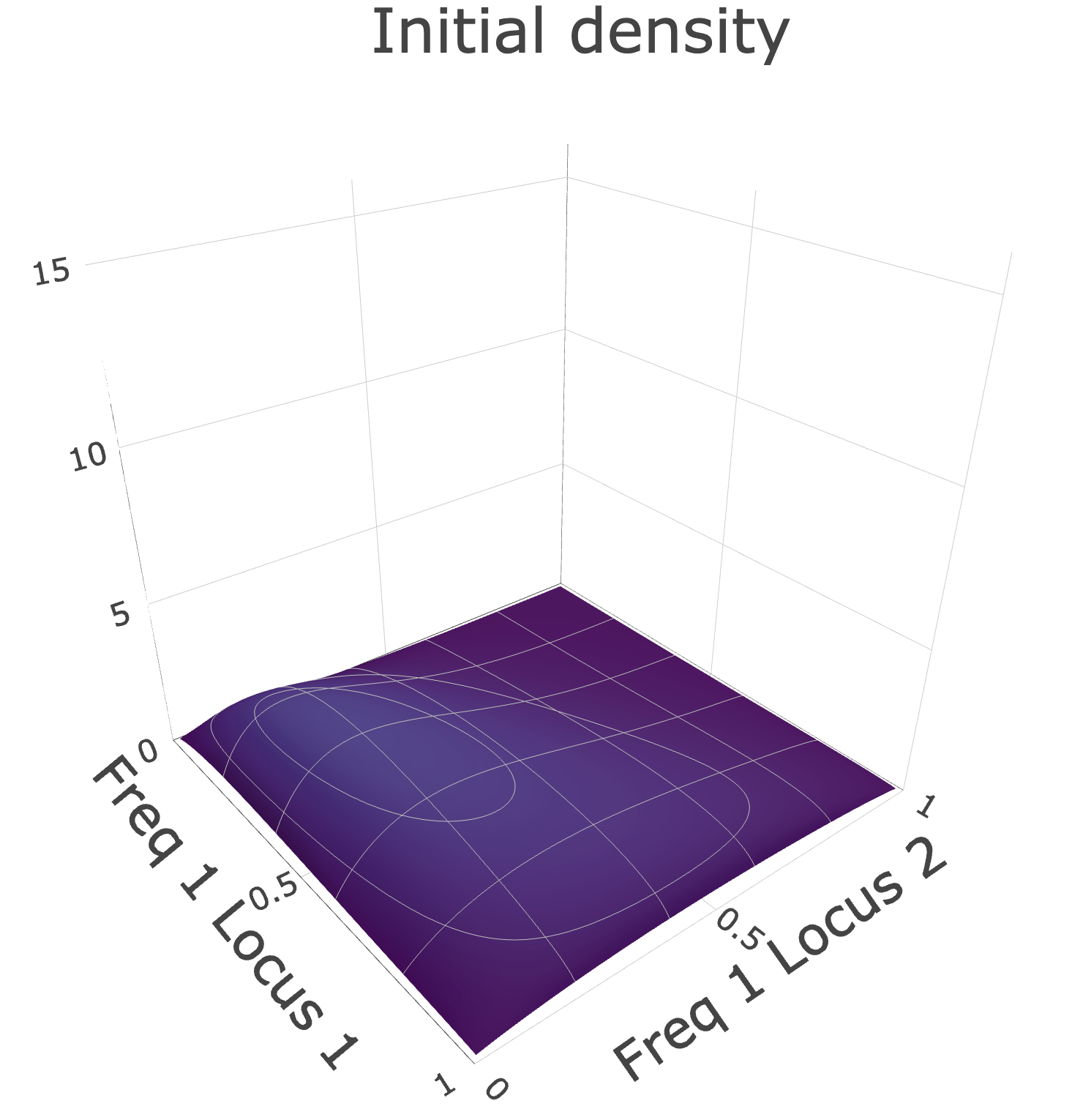}
	\else 
	\text{\red{decomment \emph{figurestrue} in preamble}} \fi
    \end{subfigure}
    \hspace{1.5cm}
    \begin{subfigure}[t]{0.28\textwidth}
        \iffigures\includegraphics[width=\linewidth]{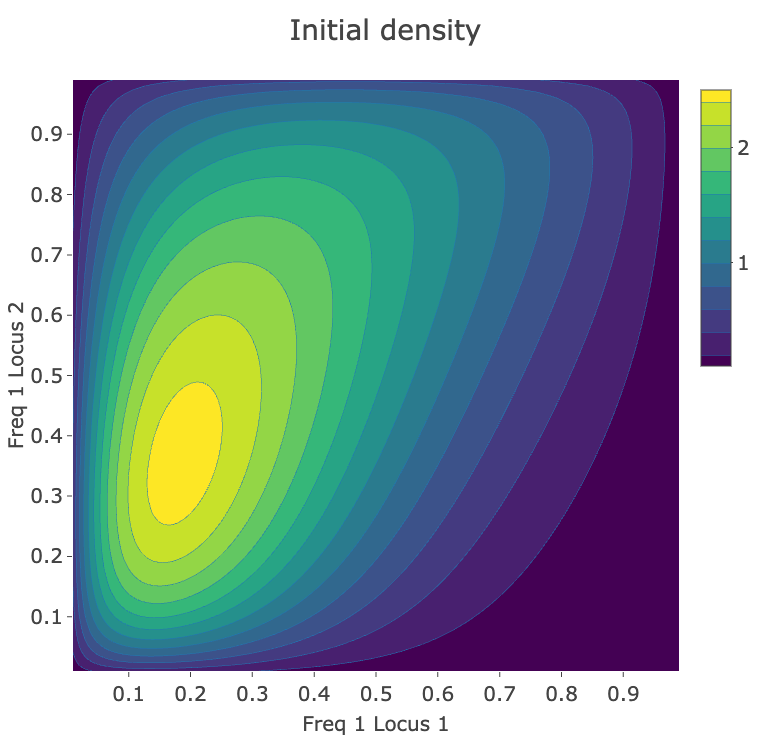}
    \else 
	\text{\red{decomment \emph{figurestrue} in preamble}} \fi
    \end{subfigure}
\vspace{0.2cm}

    \centering
    \begin{subfigure}[t]{0.28\textwidth}
        \iffigures\includegraphics[width=\linewidth]{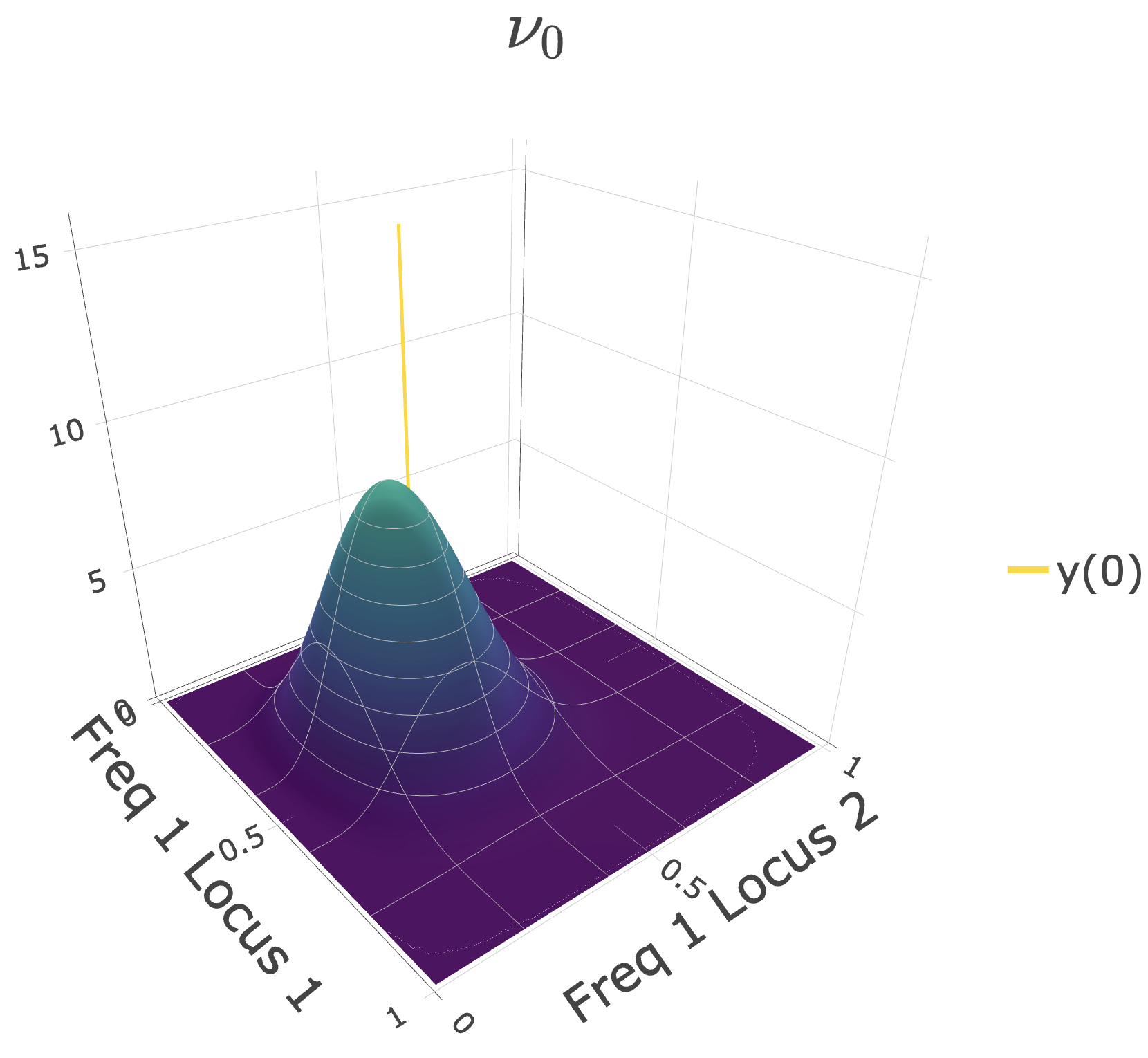}
	\else 
	\text{\red{decomment \emph{figurestrue} in preamble}} \fi
    \end{subfigure}
    \hspace{1.5cm}
    \begin{subfigure}[t]{0.28\textwidth}
        \iffigures\includegraphics[width=\linewidth]{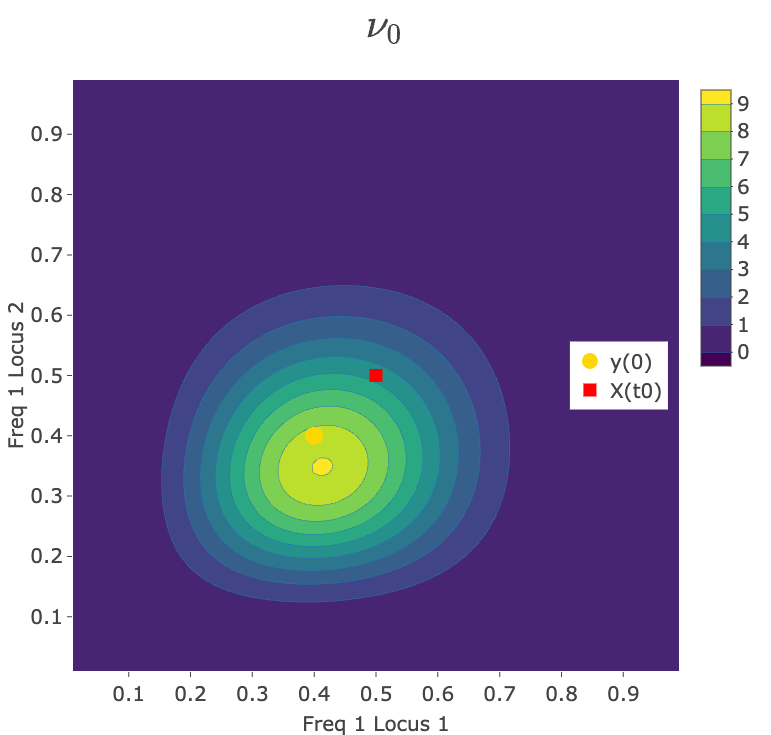}
    \else 
	\text{\red{decomment \emph{figurestrue} in preamble}} \fi
    \end{subfigure}
\vspace{0.2cm}

    \centering
    \begin{subfigure}[t]{0.28\textwidth}
        \iffigures\includegraphics[width=\linewidth]{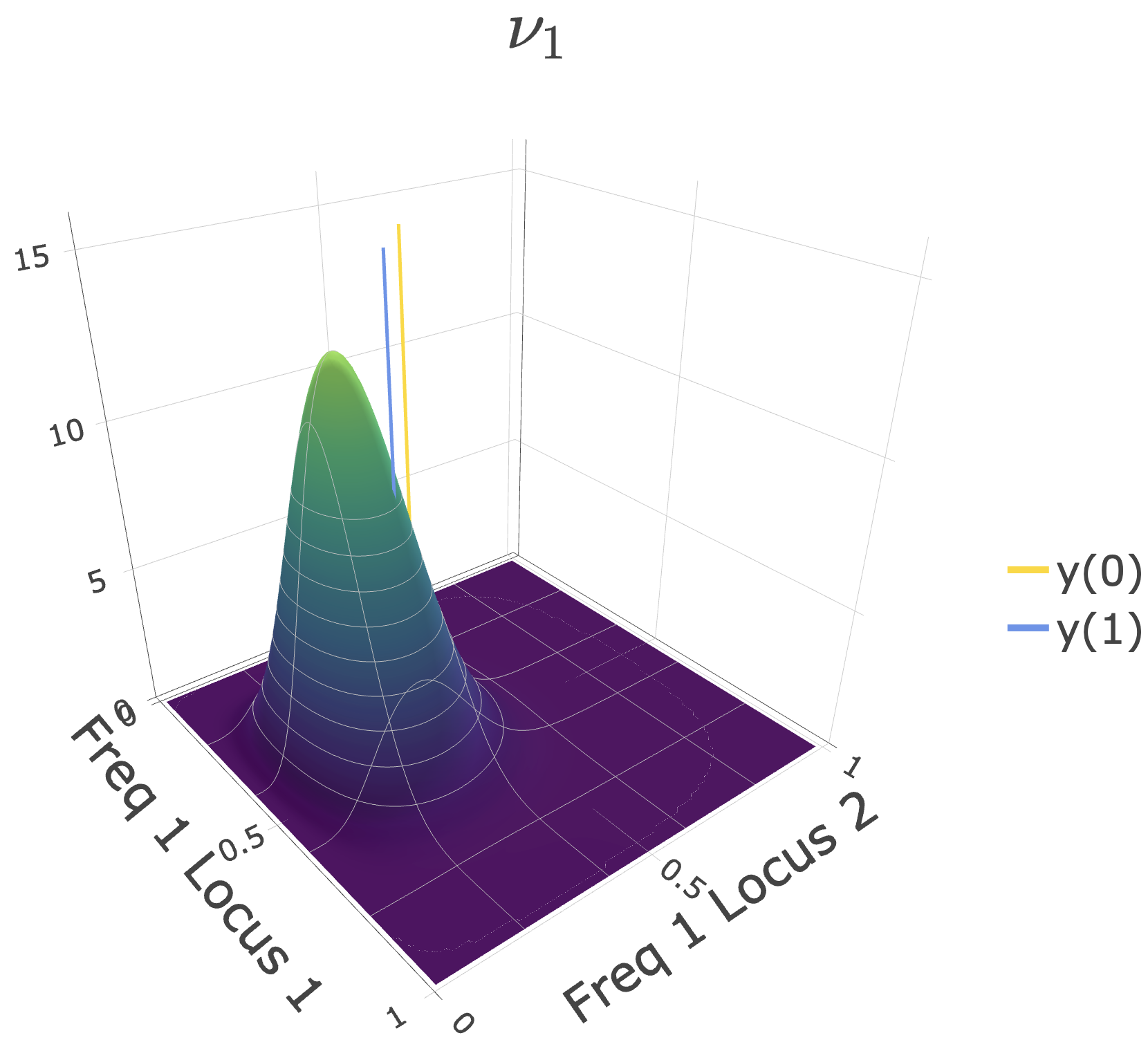}
	\else 
	\text{\red{decomment \emph{figurestrue} in preamble}} \fi
    \end{subfigure}
    \hspace{1.5cm}
    \begin{subfigure}[t]{0.28\textwidth}
        \iffigures\includegraphics[width=\linewidth]{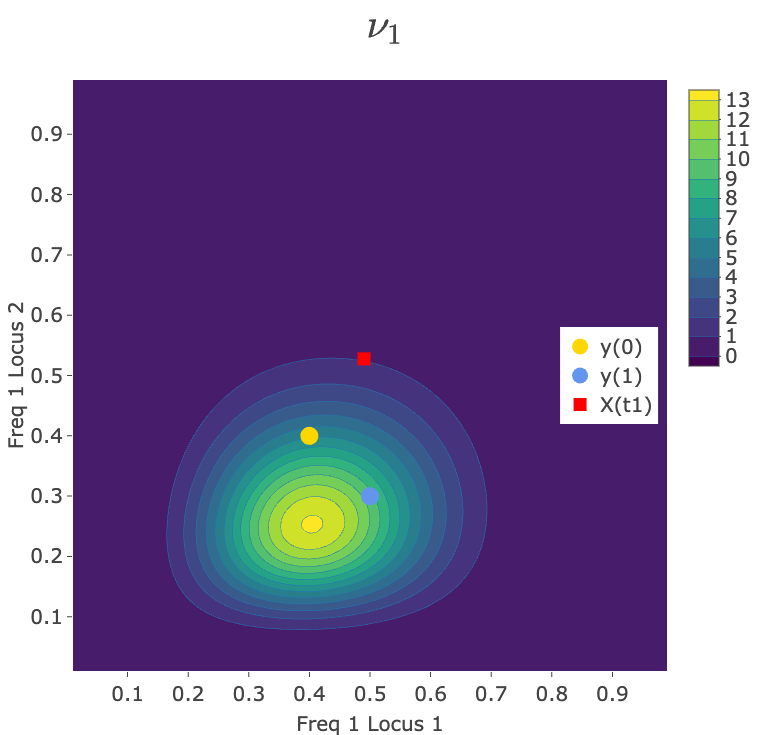}
    \else 
	\text{\red{decomment \emph{figurestrue} in preamble}} \fi
    \end{subfigure}
\vspace{0.2cm}

    \centering
    \begin{subfigure}[t]{0.28\textwidth}
        \iffigures\includegraphics[width=\linewidth]{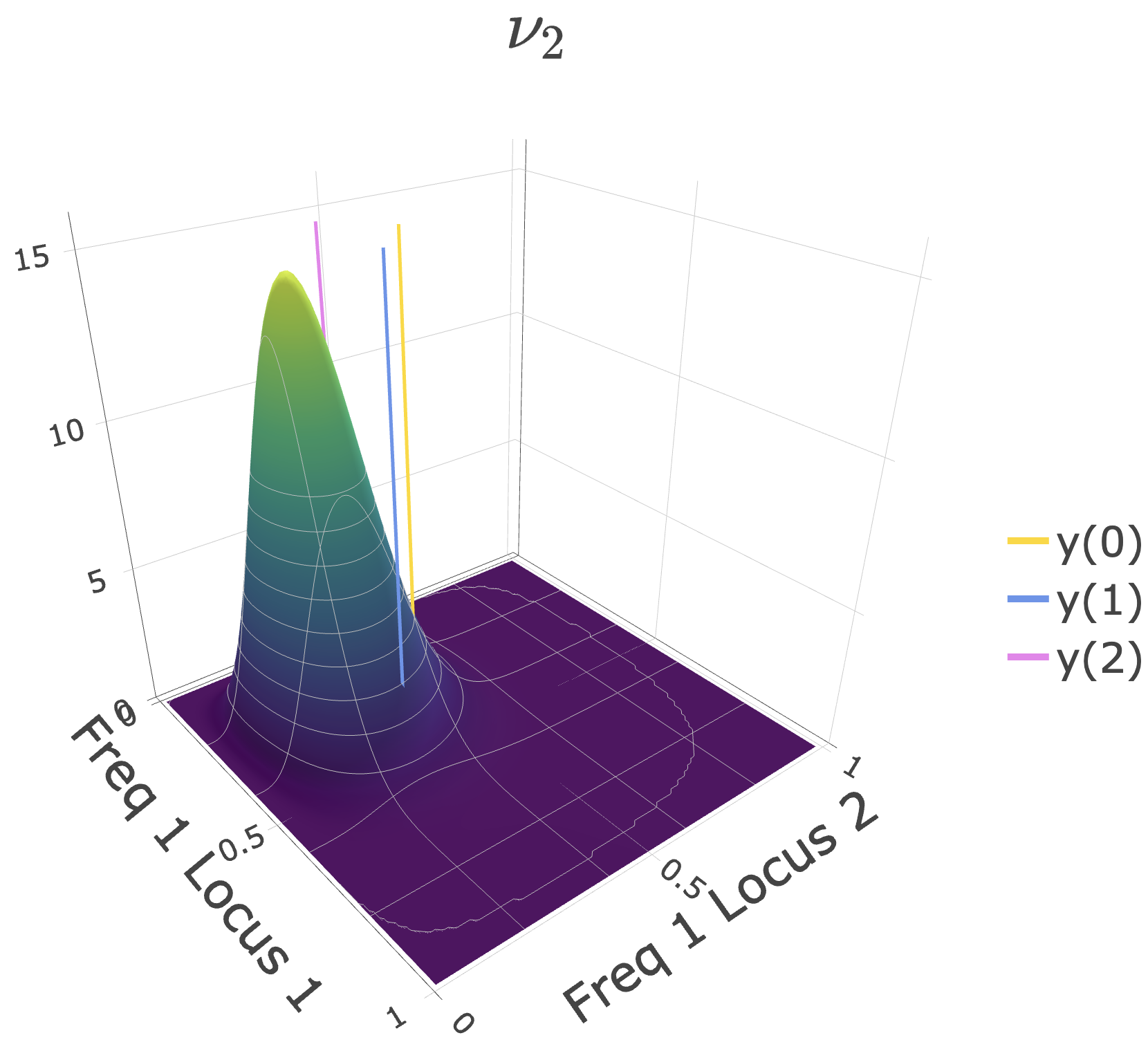}
	\else 
	\text{\red{decomment \emph{figurestrue} in preamble}} \fi
    \end{subfigure}
    \hspace{1.5cm}
    \begin{subfigure}[t]{0.28\textwidth}
        \iffigures\includegraphics[width=\linewidth]{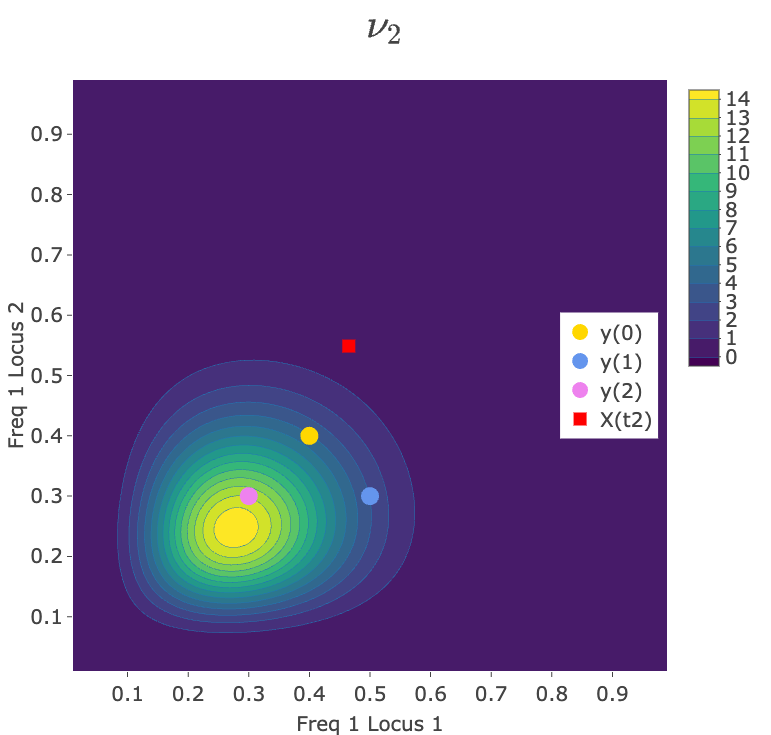}
    \else 
	\text{\red{decomment \emph{figurestrue} in preamble}} \fi
    \end{subfigure}
    
    \caption{\scriptsize 3D-plots and contours for the initial density (top row) and the filtering distributions $\nu_n$ for $n=0,1,2$. Bars (left) and dots (right) indicate the data (transformed into relative frequencies), the square (right) indicated the true signal value.}\label{fig:init-dens-nu}
\end{figure}

Figure \ref{fig:init-dens-nu} displays the initial density and the filtering distributions at three consecutive time points, for data values given by (4,6,4,6), (5,5,3,7), and (3,7,3,7), respectively. In the top row, the initial distribution, before collecting any data, is fairly well spread on the state space, with a slight skew towards lower frequencies of type 1 at either locus.
The subsequent three rows illustrate the filtering distributions computed according to the strategy outlined in Proposition \ref{prop:filtering}, using a naive Monte Carlo approximation of the required normalising constant and a Gillespie-based approximation of the dual process transition probabilities. 
Bars in left plots and dots in right plots indicate the data, transformed into relative frequencies of type 1, while the red square indicates the true signal value. The Figure shows how the filtering distribution uses the information present in the data to move the mass towards regions of the state space where the signal is believed to be.

The approach to computing the filtering and smoothing distributions described and illustrated here is not limited to this specific example, but  can in principle be applied to any parameterization of the model considered in this paper, for any number of types and loci. The filtering and smoothing distributions obtained through this method can then be used to derive estimators for the signal values at the time points of interest, as well as an approximate representation of the likelihood given a dataset which in turn is the basis for parameter inference. A comprehensive investigation of full inference for this model, together with an evaluation of the algorithm's performance, remain topics of great interest, which we leave for future work.


\section*{Acknowledgements}

The authors are grateful to two anonymous Referees for carefully reading the manuscript and providing constructive comments.
MR is partially supported by the European Union - Next Generation EU funds, PRIN-PNRR 2022 (P2022H5WZ9).


\section*{Declarations}

The authors have no competing interests to declare that are relevant to the content of this article.


\end{document}